\documentclass{jgcc}

\pdfoutput=1
\usepackage{lastpage}
\jgccdoi{13}{1}{3}{7347}  % for volume 15, issue 1, 3rd paper, permanent id (on episciences) 6059
\jgccheading{}{\pageref{LastPage}}{}{}%
{Apr.~11,~2021}{Aug.~25,~2021}{}
 
\usepackage[utf8]{inputenc}
\usepackage{amsmath}
\usepackage{amsthm}
\usepackage{amsfonts}
\usepackage{amssymb}
\usepackage{hyperref}

\let\leq\leqslant

\let\geq\geqslant
\let\epsilon\varepsilon
\let\phi\varphi

\begin{document}

\title{Groups with context-free Diophantine problem}
\author{Vladimir Yankovskiy}
\address{Faculty of Mechanics and Mathematics of Moscow State University, Moscow 119991 Russia, Leninskie Gory, MSU. Moscow Center for Fundamental and Applied Mathematics, Russia}	
\email{vladimir\_yankovskiy@mail.ru} 
\thanks{The work of the author was supported by the Russian Foundation for Basic Research, project no. 19-01-00591, and the Ministry of Education and Science of the Russian Federation as part of the program of the Moscow Center for Fundamental and Applied Mathematics under the agreement 075-15-2019-1621.}	

\begin{abstract}

We find algebraic conditions on a group equivalent to the position
of its Diophantine problem in the Chomsky Hierarchy.
In particular,
we prove that a finitely generated group has a context-free Diophantine problem if and only if it is finite.


\end{abstract}

\maketitle

\section{Introduction}

The main result of this article is the following:

\begin{thm}\label{thm:main-one}
Suppose that $G$ is a finitely generated group, and $A$ is a finite symmetric generating set for $G$. Then the following statements are equivalent:
\begin{itemize}
\item The $n$-ary Diophantine problem is context-free for some $n \geq 1$;
\item The $n$-ary Diophantine problem is context-free for all $n \geq 1$;
\item $G$ is finite.
\end{itemize}
\end{thm}

Similar results are proved in Section 3 for groups with regular Diophantine problems (they are exactly all finite groups) and for groups with recursively enumerable Diophantine problems (they are exactly all finitely generated subgroups of finitely presented groups).

The Diophantine problem is a generalization of the classical word problem, first introduced by Dehn in \cite{dehn}. Classification of groups with recursively enumerable word problem was obtained by Higman in \cite{higman}, classification of groups with regular word problem was obtained by Anisimov in \cite{anisimov} and classification of groups with context-free word problem was obtained by Muller and Schupp in \cite{shupp}. Those results will be formulated in Section 2.

Equations over groups have previously been studied from a language-theoretic point of view. Ciobanu, Diekert and Elder proved that the full sets of solutions to systems of equations in a free group form EDT0L languages, and that their specification can be computed in $\mathsf{NSPACE}(n \log n)$ \cite{cde}. This was followed by Diekert,  Je\.z and Kufleitner for right-angled Artin groups \cite{djk}, Diekert and Elder for virtually free groups \cite{diel}, Ciobanu and Elder for hyperbolic groups \cite{ciel}, and Evetts and Levine for virtually abelian groups \cite{evle}; in each case the language of solutions is shown to be EDT0L and specifications can be computed with various space bounds. It was also recently shown by Ferens and Je\.z \cite{feje} that full sets of solutions to systems of unary equations (equations with only one variable) in free groups form regular languages, and that their specification can be computed in cubic time.

However, our approach to studying languages associated with equations differs from this previous work, where the languages consist of tuples of words that represent values that can be assigned to variables in an equation or system of equations which solve those equations. On the contrary, here we consider the language consisting of the actual equations themselves which have solutions.

This article is organised as follows.
Section 2 deals with the definition of polynomials on groups and with language theoretic properties of the word problem.
Section 3 deals with the definition of the Diophantine problem and with the necessary and sufficient conditions for it to be regular and recursively enumerable.
Sections 4 and 5 contain Lemmas used in the proof of Theorem~\ref{thm:main-one}.
In Section 6 we prove Theorem~\ref{thm:main-one}.

\section{Polynomials in groups}

\begin{defi}
Suppose that $G$ is a finitely generated group, $A$ is a finite symmetric generating set for $G$, and $\{x_1, x_1^{-1}, ... , x_n, x_n^{-1}\}$ is a set of \emph{formal variables} and their inverses. An {\em $n$-ary polynomial over $(G, A)$} is a word from $P_n(G, A) := (A \cup \{x_1, x_1^{-1}, ... , x_n, x_n^{-1}\})^*$.
\end{defi}

Now we define an additional structure on the polynomials:

\begin{defi}
Suppose that $G$ is a finitely generated group, and $A$ is a finite symmetric generating set for $G$. Then {\em the interpretation of $n$-ary polynomials over $(G, A)$ } is the function $i: P_n(G, A) \to G^{G^n}$, defined recursively by:
\begin{align*}
i(\Lambda)(g_1, ... , g_n) &= e\qquad \text{and, for all $\alpha \in P_n(G, A)$,} \\
i(a \alpha)(g_1, ... , g_n) &= a i(\alpha)(g_1, ... , g_n), \\
i(x_k \alpha)(g_1, ... , g_n) &= g_k i(\alpha)(g_1, ... , g_n), \\
i(x_k^{-1} \alpha)(g_1, ... , g_n) &= g_k^{-1} i(\alpha)(g_1, ... , g_n).
\end{align*}
Here and thereafter $\Lambda$ stands for the empty word and $e$ stands for the identity of $G$.
\end{defi}

\begin{defi}
Suppose that $G$ is a finitely generated group and $A$ is a finite symmetric generating set for $G$. The {\em  word problem in $(G, A)$} is the formal language $Id(G, A) = \{\alpha \in P_0(G, A) \mid i(\alpha) = e\}$.
\end{defi}

 \begin{thm}{{\em \cite{anisimov}}}\label{thm:anisimov}
 $Id(G, A)$ is regular if and only if $G$ is finite.
 \end{thm}
 
  \begin{thm}{{\em \cite{shupp}}}\label{thm:schupp}
$Id(G, A)$ is context-free if and only if $G$ is virtually free.
 \end{thm}
 
  \begin{thm}{{\em \cite{higman}}}\label{thm:higman}
$Id(G, A)$ is recursively enumerable if and only if $G$ is a subgroup of a finitely presented group.
 \end{thm}

\section{Diophantine problem}

\begin{defi}
We call $(g_1, ..., g_n) \in G^n$ {\em a solution of the $n$-ary polynomial $\alpha \in P_n(G, A)$ }  if $i(\alpha)(g_1, ... , g_n) = e$.
\end{defi}

\begin{defi}
The {\em $n$-ary Diophantine problem over $(G, A)$} is the language $Eq_n(G, A)$ of all polynomials from  $P_n(G, A)$ that have a solution.
\end{defi}

It is not hard to see that $Eq_0(G, A) = Id(G, A)$. From this, we can derive the following necessary and sufficient conditions for its regularity and recursive enumerability.

\begin{thm}
Suppose that $G$ is a finitely generated group and $A$ is a finite symmetric generating set for $G$. Then the following statements are equivalent:
\begin{itemize}
\item $Eq_n(G, A)$ is regular for some $n \geq 1$;
\item $Eq_n(G, A)$ is regular for all $n \geq 1$;
\item $G$ is finite.
\end{itemize}
\end{thm}

\begin{proof}   
If $Eq_n(G, A)$ is regular for all $n \geq 1$, then $Eq_n(G, A)$ is regular for some $n \geq 1$.

If $Eq_n(G, A)$ is regular for some $n \geq 1$, then $Eq_n(G, A) \cap A^* = Id(G, A)$ is also regular. Thus $G$ is finite by Theorem~\ref{thm:anisimov}.

Suppose that $G$ is finite. Then to conclude the proof we build a deterministic finite automaton that recognizes $Eq_n(G, A)$ for any $n$. In our deterministic finite automaton, the set of states is the set of all functions from $G^n$ to $G$ (their number is equal to $|G|^{n|G|}$). The transition function $\psi\colon G^{G^n} \times P_n(G, A) \to G^{G^n}$ is defined by the formula
 
 \begin{equation}\psi(f, a)(g_1, ... , g_n) = \begin{cases} f(g)a \quad \text{if } a \in A \\ f(g)g_j \quad \text{if } a = x_j \\ f(g)g_j^{-1} \quad\text{if } a = x_j^{-1} \end{cases}\end{equation}
% $$\forall f \in G^{G^n}, g \in G, \alpha \in P_n(G, A), j \leq n $$
for all $f \in G^{G^n}$, $g \in G$, $\alpha \in P_n(G, A)$ and $j \leq n$.

 The initial state is the constant function on $G^n$ with value $e$. The set of terminal states is $\{\alpha \in P_n(G, A) \mid \exists g_1, ... , g_n \in G \text{ such that } i(\alpha)(g_1, ... , g_n) = e \}$.
\end{proof}

\begin{thm}
Suppose that $G$ is a finitely generated group and $A$ is a finite symmetric generating set for $G$. Then the following statements are equivalent:
\begin{itemize}
\item $Eq_n(G, A)$ is recursively enumerable for some $n \geq 1$;
\item $Eq_n(G, A)$ is recursively enumerable for all $n \geq 1$;
\item $G$ is a subgroup of a finitely presented group.
\end{itemize}
\end{thm}

\begin{proof} 
If $Eq_n(G, A)$ is recursively enumerable for all $n \geq 1$, then $Eq_n(G, A)$ is recursively enumerable for some $n \geq 1$.
If $Eq_n(G, A)$ is recursively enumerable for $n \geq 1$, then $Eq_n(G, A) \cap A^* = Id(G, A)$ is also recursively enumerable. That means $G$  is a subgroup of a finitely presented group by Theorem~\ref{thm:higman}.
Suppose that $G$ is a subgroup of a finitely presented group. That means its word problem is recursively enumerable by Theorem 4. Then the words from $Eq_n(G, A)$ are recognized by the following algorithm:

Suppose that $w$ is our word. 
On the $m$-th step, for each tuple of words $(v_1, ... v_n)  \in A^*$ of length $m$ or less, we run the first $m$ steps of the $Id(G, A)$ recognition algorithm for 
$$w[x_1:=v_1, ... , x_n := v_n].$$
If the $Id(G, A)$ recognition algorithm terminates within $m$ steps for at least one of the tuples, then $w \in Eq_n(G, A)$. Otherwise we launch step $m+1$. \end{proof}

To find the necessary and sufficient conditions for  $Eq_n(G, A)$ being context-free we first need to pay attention to certain facts that we mention in the next two Sections.

\section{Context-free sets of rational numbers}

\begin{defi}
A set of positive rational numbers $Q$ is {\em  context-free} if the language $\{a^{m}b^n \; | \; \frac{m}{n} \in Q\}$ is context-free.
\end{defi}

\begin{lem}{{\em \cite{barhillel}.}}\label{thm:pumping-lemma}
If a language $L$ is context-free, then there exists an integer $p \geq 1$, called the pumping length, such that any word $s \in L$ of length at least $p$ can be written as $s=uvwxy$, where $|vx|\geq 1$, $|vwx|\leq p$ and $\forall n \in \mathbb{N}$ $uv^{n}wx^{n}y\in L$.
\end{lem}

\begin{lem}\label{thm:main-lemma}
Any context-free set of rational numbers has finitely many isolated points.
\end{lem}

\begin{proof}
Suppose that $Q \subset \mathbb{Q}_+$ is a context-free set with infinitely many isolated points. Then $L = \{a^{m}b^n \; | \; \frac{m}{n} \in Q\}$ is context-free. Suppose that $p$ is its pumping length,  $\frac{M}{N} \in Q$ is an isolated element of $Q$ such that $\frac{M}{N}$ is irreducible and $M + N > p$ (such an element exists by the Pigeonhole Principle, because $Q$ is infinite and there are only finitely many positive rational numbers $\frac{M}{N}$, such that $M + N \leq p$),\begin{align}
\epsilon &:= \inf\left\{|q - \frac{M}{N}| \mid q \in Q\setminus \left\{ \frac{M}{N}\right\}\right\} > 0,\\
n &:= \left\lceil \frac{(M + N)^2}{\epsilon N^2} \right\rceil + 1,\\
w &:=a^{nM}b^{nN}.
\end{align}
Then by Lemma~\ref{thm:pumping-lemma} $w=uvxyz$, where $|vxy| \leq p$, $|vy| \geq 1$, and $uv^txy^tz\in L$  for all $t \in \mathbb{N}$. Suppose that $\alpha$ is the number of instances of $a$ in $vy$ and $\beta$ is the number of instances of $b$ in $vy$. If $v$ or $y$ contains both $a$ and $b$, then $uv^2xy^2z$ contains $ba$ and therefore does not belong to $L$. 
So $3$ variants are possible:
\begin{itemize}
\item If $\beta = 0$ then for all $t \in \mathbb{N}_0$ $a^{nM + (t - 1)\alpha}b^{nN} \in L$ by the Lemma~\ref{thm:pumping-lemma}. That means $\frac{nM + (t - 1)\alpha}{nN} \in Q$ for all $t \in \mathbb{N}_0$. Take $t = 2$. Then $\frac{nM + \alpha}{nN} \in Q$. But $|\frac{nM + \alpha}{nN} - \frac{M}{N}| = \frac{\alpha}{nN} \leq \frac{M + N}{nN} < \epsilon$, because  $\alpha \leq p \leq M+N$. Contradiction.
\item If $\alpha = 0$ then for all $t \in \mathbb{N}$ $a^{nM}b^{nN + (t - 1)\beta} \in L$ by the Lemma~\ref{thm:pumping-lemma}. That means $\forall t \in \mathbb{N}$ $\frac{nM}{nN + (t - 1)\beta} \in Q$ for all $t \in \mathbb{N}$. Take $t = 2$. Then $\frac{nM}{nN + \beta} \in Q$. But $|\frac{nM}{nN + \beta} - \frac{M}{N}| = \frac{(M+N)M}{((n+1)N + M)N} \leq \frac{M^2}{nN^3} < \epsilon$ because  $\alpha \leq p \leq M+N$. Contradiction.
\item If $v = a^\alpha$ and $y = b^\beta$ then $a^{nM+(t-1)\alpha}b^{nN+(t-1)\beta} \in L$ for all $t \in \mathbb{N}$ by the Lemma~\ref{thm:pumping-lemma}. That means $\frac{nM+(t-1)\alpha}{nN+(t-1)\beta} \in Q$  for all $t \in \mathbb{N}$. Take $t = 2$. Then $\frac{nM + \alpha}{nN + \beta} \in Q$.
But $|\frac{nM + \alpha}{nN + \beta} - \frac{M}{N}| = \frac{|\alpha N - \beta M|}{N(nN + \beta)} < \epsilon$. It is possible only when $\frac{\alpha}{\beta} = \frac{M}{N}$. And from that follows $p < M + N \leq \alpha + \beta = |vy| \leq |vxy| \leq p$. Contradiction.
\end{itemize}
These cases complete the proof that  $L$ is not context-free.
\end{proof}

\section{Commensuration of group elements}

\begin{defi}
Suppose $G$ is a group. Then $a, b \in G$ are {\em commensurate} if there exist $n, m \in \mathbb{Z}\setminus\{0\}$ such that $a^n = b^m$.
\end{defi}

It is not hard to see that commensuration is an equivalence relation. We denote by $CC(a)$ the commensuration class of $a$.

\begin{thm}{{\em \cite{gromov}.}}\label{thm:gromov}
Commensuration classes of hyperbolic groups are of the following types:
\begin{itemize}
\item set of all finite order elements;
\item set of all infinite order elements of a maximal virtually cyclic subgroup.
\end{itemize}
\end{thm}

\section{Proof of the Theorem~\ref{thm:main-one}}

\begin{proof}
If $Eq_n(G, A)$ is context-free for all $n \geq 1$, it is context-free for some $n \geq 1$.

If $G$ is finite then by Theorem~\ref{thm:schupp} $Eq_n(G, A)$ is context-free for all $n \geq 1$. 

If $Eq_n(G, A)$ is context-free for some $n \geq 1$, then $Eq_n(G, a) \cap P_1(G, A) = Eq_1(G, A)$ is also context-free.
Suppose now that the language $Eq_1(G, A)$ is context-free and $G$ is infinite. Then the language $Id(G, A) = Eq_1(G, A) \cap A^*$ is also context-free as intersection of a context-free language with a regular one. That means $G$ is virtually-free (and therefore hyperbolic) by the Theorem~\ref{thm:gromov}. Suppose that $a$ is an infinite-order element of $G$ and suppose that $H := \langle CC_G(a) \rangle$. Then $H$ is an virtually infinite cyclic group. Suppose that $c$ is a generator of a finite-index cyclic normal subgroup of $H$. Now $w \in P_0(G, A)$, $i(w) = c$. Consider the language $Eq_1(G, A) \cap \{w\}^*\{x\}^*$. It is context-free as it is the intersection of a context-free language with a regular one. However $Eq(G, A) \cap \{w\}^*\{x\}^* = \{w^nx^m \; | \; \exists b \in H \; b^mc^n = e\} = \{w^nx^m \; | \; \frac{n}{m} \in \{\frac{n}{k} \in \mathbb{Q} \; | \; \exists b \in H \; b^k = c^n\})\}$. But $\{\frac{n}{k} \in \mathbb{Q} \; | \; \exists b \in H \; b^k = c^n\}$ is discrete as $\langle c \rangle$ has finite index in $H$. This contradicts Lemma~\ref{thm:main-lemma}.
\end{proof}

\section*{Acknowledgements}

I thank my supervisor Anton Klyachko for providing valuable advice regarding this research.

\bibliographystyle{plain}
\bibliography{bibliography}

\begin{thebibliography}{10}

\bibitem{anisimov}
Anatoly Anisimov.
\newblock {\"U}ber gruppen-sprachen.
\newblock {\em Kibernetika}, 4:18--24, 1971.

\bibitem{barhillel}
Yehoshua Bar-Hillel, Micha Perles, and Eli Shamir.
\newblock On formal properties of simple phrase-structure grammars.
\newblock {\em Zeitschrift f{\"u}r Phonetik, Sprachwissenschaft, und
  Kommunikationsforschung}, 14:143--172, 1961.

\bibitem{cde}
Laura Ciobanu, Volker Diekert, and Murray Elder.
\newblock Solution sets for equations over free groups are {EDT}0{L} languages.
\newblock {\em Internat. J. Algebra Comput.}, 26(5):843--886, 2016.

\bibitem{ciel}
Laura Ciobanu and Murray Elder.
\newblock Solutions sets to systems of equations in hyperbolic groups are
  {EDT}0{L} in $\mathsf{PSPACE}$.
\newblock In {\em 46th {I}nternational {C}olloquium on {A}utomata, {L}anguages,
  and {P}rogramming}, volume 132 of {\em LIPIcs. Leibniz Int. Proc. Inform.},
  pages Art. No. 110, 15. Schloss Dagstuhl. Leibniz-Zent. Inform., Wadern,
  2019.

\bibitem{dehn}
Max Dehn.
\newblock {\"U}ber unendliche diskontinuierliche gruppen.
\newblock {\em Mathematische Annalen}, 71:116--144, 1911.

\bibitem{diel}
Volker Diekert and Murray Elder.
\newblock Solutions to twisted word equations and equations in virtually free
  groups.
\newblock {\em Internat. J. Algebra Comput.}, 30(4):731--819, 2020.

\bibitem{djk}
Volker Diekert, Artur Je\.{z}, and Manfred Kufleitner.
\newblock Solutions of word equations over partially commutative structures.
\newblock In {\em 43rd {I}nternational {C}olloquium on {A}utomata, {L}anguages,
  and {P}rogramming}, volume~55 of {\em LIPIcs. Leibniz Int. Proc. Inform.},
  pages Art. No. 127, 14. Schloss Dagstuhl. Leibniz-Zent. Inform., Wadern,
  2016.

\bibitem{evle}
Alex Evetts and Alex Levine.
\newblock Equations in virtually abelian groups: languages and growth.
\newblock {\em CoRR}, abs/2009.03968, 2020.

\bibitem{feje}
Robert Ferens and Artur Je\.z.
\newblock Solving one variable word equations in the free group in cubic time.
\newblock In Markus Bl{\"{a}}ser and Benjamin Monmege, editors, {\em 38th
  International Symposium on Theoretical Aspects of Computer Science, {STACS}
  2021, March 16-19, 2021, Saarbr{\"{u}}cken, Germany (Virtual Conference)},
  volume 183 of {\em LIPIcs}, pages 30:1--30:17. Schloss Dagstuhl -
  Leibniz-Zentrum f{\"{u}}r Informatik, 2021.

\bibitem{gromov}
Mikhael Gromov.
\newblock Hyperbolic groups.
\newblock {\em Math. Sci. Res. Inst. Publ.}, 8:75--263, 1987.

\bibitem{higman}
Graham Higman.
\newblock Subgroups of finitely presented groups.
\newblock {\em Proceedings of the Royal Society, Series A, Mathematical and
  Physical Sciences}, 262:455--475, 1961.

\bibitem{shupp}
David~E. Muller and Paul~E. Schupp.
\newblock Groups, the theory of ends, and context-free languages.
\newblock {\em J. Comput. Syst. Sci.}, 26(3):295--310, 1983.

\end{thebibliography}

\typeout{get arXiv to do 4 passes: Label(s) may have changed. Rerun}
\end{document}